\documentclass{amsart}

\usepackage{amsmath, amssymb, amsthm, mathtools}
\usepackage{tikz}
\usetikzlibrary{arrows.meta}
\usepackage{tikz-cd}
\usepackage[hidelinks]{hyperref}
\usepackage{lmodern}
\usepackage{microtype}
\theoremstyle{plain}
\newtheorem{theorem}{Theorem}[section]
\newtheorem{proposition}[theorem]{Proposition}
\newtheorem{lemma}[theorem]{Lemma}
\newtheorem{corollary}[theorem]{Corollary}

\theoremstyle{definition}
\newtheorem{definition}[theorem]{Definition}
\newtheorem{remark}[theorem]{Remark}

\newcommand{\K}{\mathcal{K}}
\newcommand{\Hom}{\mathrm{Hom}}
\newcommand{\Hd}{\mathrm{Hd}}
\newcommand{\llim}{\varprojlim}
\newcommand{\llo}{\mathop{{\textstyle\varprojlim}^{1}}}
\newcommand{\hlim}{\mathrm{h}\text{-}\!\varinjlim}
\newcommand{\ev}{\mathrm{ev}}
\newcommand{\id}{\mathrm{id}}
\DeclareMathOperator{\sh}{sh}
\newcommand{\Ndag}{\mathbb{N}^{\dagger}}
\newcommand{\asto}{\xrightarrow{\;\approx\;}}
\newcommand{\cz}{\overline{\{0\}}}

\newif\ifanon
\anonfalse

\title{A Milnor exact sequence for $E$-theory}
\ifanon\else
\author{Jos\'e R.\ Carri\'on}
\address{Department of Mathematics, Texas Christian University, Fort Worth, TX 76129, USA}
\email{j.carrion@tcu.edu}
\thanks{The author was partially supported by NSF grant DMS-2451675.}
\fi

\begin{document}

\begin{abstract}
	For separable $C^*$-algebras $A$ and $B$, the $E$-theory group $E(A,B)$ carries a natural, generally non-Hausdorff, second-countable group topology, Hausdorff exactly when the closure of zero $\cz \subseteq E(A,B)$ is trivial.
	The nonzero elements of $\cz$ are the phantom classes, invisible to every continuous Hausdorff-valued invariant.
	We identify $\cz$ as a derived inverse limit over any shape system $(C_n)$ of $SA \otimes \K$, \( \cz \cong \llo [C_n,S^2B \otimes \K], \) giving a natural Milnor sequence
	\[
		0 \to \llo [C_n,S^2B \otimes \K] \to E(A,B) \to \llim [C_n,SB \otimes \K] \to 0
	\]
	for all separable $A$ and $B$, with no UCT or nuclearity hypothesis.
	For nuclear $A$ it agrees with the Willett--Yu controlled-$KK$ Milnor sequence, and when $A$ satisfies the UCT, nuclear or not, its $\varprojlim^1$ term is the pure-extension group $\mathrm{Pext}^1_{\mathbb Z}(K_{*+1}(A),K_*(B))$, as in Schochet's fine-structure description of the Kasparov groups.
\end{abstract}

\maketitle

\section{Introduction}

$E$-theory is a bivariant homology theory built from asymptotic morphisms.
For separable $A$, $B$ it is $E(A,B) = [[SA \otimes \K,SB \otimes \K]]$, a group of asymptotic-homotopy classes of asymptotic morphisms, and it agrees with Kasparov's $KK(A,B)$ when $A$ is nuclear \cite{Connes-Higson90}.
In \cite{Carrion-Schafhauser23} a natural second-countable group topology is placed on $E(A,B)$, the $E$-theoretic analog of Dadarlat's topology on $KK$ \cite{Dadarlat05}, with which it agrees when $A$ is nuclear: a sequence $x_m \to x$ if and only if there is $y \in E(A, C(\Ndag, B))$ with $y(m) = x_m$ and $y(\infty) = x$ (known as Pimsner's condition), where $\Ndag = \mathbb N \cup \{\infty\}$.
The topology is in general not Hausdorff.
It is Hausdorff exactly when the closure of zero $\cz \subseteq E(A,B)$ is trivial, and the Hausdorffized group \( EL(A,B) \coloneqq E(A,B)/\cz \) is the $E$-theoretic counterpart of R{\o}rdam's $KL$ \cite{Rordam95, Dadarlat05}.
For nuclear $A$ one has $EL(A,B) = KL(A,B)$ \cite[Section~4.2]{Carrion-Schafhauser23}.

R{\o}rdam introduced $KL$ for $A$ satisfying the UCT as the quotient of $KK(A,B)$ by the image of the pure extensions, because approximate unitary equivalence, central to the classification of nuclear $C^*$-algebras, detects a class only up to that subgroup \cite{Rordam95}.
That image coincides with the closure of zero when $A$ is nuclear and satisfies the UCT, and $KL(A,B) = KK(A,B)/\cz$ became the general definition \cite{Dadarlat05}.
Thus $\cz$ is the information these approximate comparisons leave unresolved.

The nonzero elements of $\cz$ are the \emph{phantom} classes: $E$-theory classes that are not zero yet are identified with $0$ by every continuous, Hausdorff-valued invariant.
The name is borrowed from homotopy theory, where a phantom map is nontrivial but restricts to a null-homotopic map on every finite subcomplex, and where the phantom maps between two spaces are themselves identified with a $\varprojlim^1$ (cf.\ \cite[Section~12]{Schochet03}).
Theorem~\ref{thm:main}, the main result of this paper, is an $E$-theoretic analog of that description, and Corollary~\ref{cor:phantom} makes the analogy precise: the phantom classes are exactly the classes that vanish on every stage of a shape system of $SA \otimes \K$.

In the $KK$-picture this subgroup has been described as a $\varprojlim^1$ in several settings.
The first description of this kind goes back to work of L.~Brown on $\mathrm{Ext}$ \cite{Brown75, Brown77b, Brown84}.
When $A$ satisfies the UCT, Schochet \cite{Schochet96, Schochet98} constructed from a $KK$-filtration of $A$ (a filtration that exists precisely when the UCT holds) a Milnor sequence for $KK(A,B)$ whose $\varprojlim^1$ term is the pure-extension group $\mathrm{Pext}^1_{\mathbb Z}(K_{*+1}(A), K_*(B))$.
The fine-structure papers \cite{Schochet01, Schochet02} identify $\cz$ with that term.
Willett and Yu \cite{Willett-Yu21} identified $\cz$ with a $\varprojlim^1$ of their controlled $KK$-groups.

We work in $E$-theory instead.
For ($K$-)nuclear $A$ one has $E(A,B) = KK(A,B)$ \cite[Corollary~9]{Connes-Higson90}, but in general the two theories differ, and it is $E$-theory that is half-exact in both variables for all separable $C^*$-algebras.
Half-exactness and countable additivity already give $E$-theory a Milnor sequence of the classical kind: if $A$ is presented as a direct limit $\varinjlim A_n$ of separable $C^*$-algebras, there is a natural exact sequence $0 \to \llo E(A_n, SB) \to E(A,B) \to \llim E(A_n,B) \to 0$ \cite[Proposition~7.2]{GHT00}, computing $E(A,B)$ from the chosen presentation (and making no reference to the topology on $E(A, B)$).
Applied to a shape system $(C_n)$ of $SA \otimes \K$, with second variable $SB$ so that the middle term is again $E(A,B)$, it identifies its $\varprojlim^1$ term with the kernel of the restriction map $E(A,B) \to \llim E(C_n,SB)$, the group of classes whose restrictions vanish in $E(C_n,SB)$.
Membership in $\cz$ asks more: the restrictions must vanish already as asymptotic-homotopy classes of $^*$-homomorphisms $C_n \to SB \otimes \K$, before suspension and stabilization (Corollary~\ref{cor:phantom}).
Working directly with the homotopy classes over a shape system of $SA \otimes \K$, we identify $\cz \subseteq E(A,B)$ as a $\varprojlim^1$ for arbitrary separable $A$ and $B$ (Theorem~\ref{thm:main}).
For non-nuclear $A$ this subgroup has not, to our knowledge, been described before.
The resulting Milnor sequence is consistent with the descriptions above.
For nuclear $A$, where $E(A,B) = KK(A,B)$, it and the Willett--Yu sequence describe the same extension of $KL$ by $\cz$ (Proposition~\ref{prop:compare}).
Whenever $A$ satisfies the UCT, nuclear or not, its $\varprojlim^1$ term is a pure-extension group of the $K$-theory of $A$ and $B$ (Proposition~\ref{prop:uct}).

Willett and Yu's sequence has since proved fruitful: they used it to give a sufficient condition for the UCT \cite{Willett-Yu24}.
The present description points the other way.
For nuclear $A$, the $\llo$ term of Theorem~\ref{thm:main} computes $\overline{\{0\}}_{KK(A,B)}$ from a shape system alone, with no appeal to the universal coefficient theorem.
A nuclear $A$ whose $\llo$ term is not isomorphic to $\mathrm{Pext}^1_{\mathbb Z}(K_{*+1}(A), K_*(B))$, for even one $B$, would therefore be a counterexample to the UCT.
We do not know whether the invariant can be made effective enough to exhibit such an algebra; we record only that it describes the obstruction without presupposing what it might detect in practice.

Recall \cite{Blackadar85} that a \emph{shape system} for a separable $C^*$-algebra $C$ is an inductive system $(C_n, \gamma_n)$ of separable $C^*$-algebras with semiprojective connecting maps and $\varinjlim C_n \cong C$---every separable $C^*$-algebra admits one \cite[Theorem~4.3]{Blackadar85}.
We write $[C,D]$ for the set of homotopy classes of $^*$-homomorphisms $C \to D$ (an abelian group when $D$ is a suspension of a stable algebra; see Lemma~\ref{lem:loops}).
Our main result is the following.

\begin{theorem}\label{thm:main}
	Let $A$ and $B$ be separable $C^*$-algebras and let $(C_n, \gamma_n)$ be a shape system for $SA \otimes \K$.
	Then the closure of zero in $E(A,B)$ is a first derived inverse limit,
	\begin{equation}\label{eq:czlim}
		\cz \cong \llo [C_n,S^2B \otimes \K].
	\end{equation}
	More precisely, there is a natural short exact sequence of abelian groups,
	\[
		0 \longrightarrow \llo [C_n,S^2B \otimes \K]
		\xrightarrow{\ \partial\ }
		E(A,B) \xrightarrow{\ \Theta\ } \llim [C_n,SB \otimes \K] \longrightarrow 0,
	\]
	in which $\Theta$ is the composition of the quotient $E(A,B) \to EL(A,B)$ with an isomorphism $EL(A,B) \cong \llim [C_n,SB \otimes \K]$ (see Proposition~\ref{prop:ELlim}), and the isomorphism~\eqref{eq:czlim} is implemented by $\partial$, regarded as a map onto $\ker\Theta = \cz$.
\end{theorem}

The identification made in the right-most term in the exact sequence above and the surjectivity of $\Theta$ are essentially the Hausdorffization theorem of \cite{Carrion-Schafhauser23} (recalled here as Theorem~\ref{thm:cs}\eqref{cs:haus}); see Lemma~\ref{lem:compare-towers}.
The new content is the connecting map $\partial$ out of $\varprojlim^1$ and the resulting identification of the closure of zero.

The proof works entirely at the level of asymptotic morphisms, using Dadarlat's homotopy limit \cite{Dadarlat94} and the Pimsner condition of \cite{Carrion-Schafhauser23} in place of the homotopy theory of the mapping spaces $\Hom(C_n,SB \otimes \K)$: with the point-norm topology the restriction maps between these spaces need not be fibrations, so the classical Milnor sequence for $\pi_0$ of the limit of a tower of fibrations is not directly available.

Theorem~\ref{thm:main} is a special case of a statement about the bifunctor $[[\,\cdot\,,\,\cdot\,]]$ on asymptotic morphisms (Theorem~\ref{thm:general}).
$E$-theory is recovered by taking source $SA \otimes \K$ and target $SB \otimes \K$.

Section~\ref{sec:prelim} recalls the framework of \cite{Carrion-Schafhauser23} and proves the comparison of towers used to pass between homotopy and asymptotic-homotopy classes.
Section~\ref{sec:milnor} constructs the connecting map and proves Theorems~\ref{thm:main} and \ref{thm:general}.
Section~\ref{sec:examples} treats naturality, a Hausdorffness criterion, the dichotomy, and the comparison with the sequences of Willett--Yu and Schochet.

\section{Preliminaries}\label{sec:prelim}

We recall the framework of \cite{Carrion-Schafhauser23}.
We refer the reader to \cite{Connes-Higson90, Blackadar98} for more information on asymptotic morphisms and $E$-theory and to \cite{Blackadar85, Dadarlat94} for semiprojectivity, shape systems, and the homotopy limit.

Throughout, $C$ and $D$ are $C^*$-algebras, with $C$ separable whenever it is the source of an asymptotic morphism.
An \emph{asymptotic morphism} $\phi \colon C \asto D$ is a family of self-adjoint linear maps $(\phi_t)_{t \ge 0}$, point-norm continuous in $t$, with $\|\phi_t(xy) - \phi_t(x)\phi_t(y)\| \to 0$ for all $x, y \in C$.
Two asymptotic morphisms $\phi, \psi \colon C \asto D$ are \emph{equivalent} if $\|\phi_t(x) - \psi_t(x)\| \to 0$ for all $x \in C$, and \emph{asymptotically homotopic} if there is an asymptotic morphism $C \asto C([0,1],D)$ whose evaluations at $0$ and $1$ are equivalent to $\phi$ and $\psi$.
We write $[[C,D]]$ for the set of asymptotic-homotopy classes and, following \cite{Carrion-Schafhauser23}, $\mathrm{H}(C,D) \subseteq [[C,D]]$ for the image of $\Hom(C,D)$ in $[[C,D]]$: genuine $^*$-homomorphisms modulo asymptotic homotopy, carrying the quotient topology from the point-norm topology on $\Hom(C,D)$.
There is no reason for the natural surjection $[C,D] \to \mathrm{H}(C,D)$ from genuine homotopy classes to be injective.
We compare the two more systematically in Lemma~\ref{lem:compare-towers}.
A $^*$-homomorphism $\rho$ with domain $D$ acts by post-composition, $\rho_*[[\phi]] = [[\rho \circ \phi]]$, and a $^*$-homomorphism $\sigma$ with codomain $C$ by pre-composition, $\sigma^*[[\phi]] = [[\phi \circ \sigma]]$.
More generally, asymptotic morphisms compose, up to asymptotic homotopy, making the sets $[[C,D]]$ the morphism sets of a category; see \cite{Connes-Higson90} or \cite[Section~25.3]{Blackadar98}.
A \emph{reparametrization} of $\phi$ is $(\phi_{r(t)})_t$ for a continuous $r \colon \mathbb R_+ \to \mathbb R_+$ with $r(t) \to \infty$.
Every asymptotic morphism is asymptotically homotopic to each of its reparametrizations \cite[25.1.2(h)]{Blackadar98}.
If $\sigma \colon C_0 \to C$ is a \emph{semiprojective} $^*$-homomorphism \cite[Definition~2.10]{Blackadar85} with $C_0$ separable, then $\phi \circ \sigma$ is asymptotically homotopic to a genuine $^*$-homomorphism for every asymptotic morphism $\phi \colon C \asto D$, so that $\sigma^*$ maps $[[C,D]]$ into $\mathrm{H}(C_0,D)$ \cite[Lemma~2.2]{Carrion-Schafhauser23}.

\begin{definition}[{\cite[Definition~2.3]{Carrion-Schafhauser23}}]\label{def:top}
	For $C$ separable, $[[C,D]]$ carries the weakest topology making $\sigma^* \colon [[C,D]] \to \mathrm{H}(C_0,D)$ continuous for every separable $C_0$ and every semiprojective $^*$-homomorphism $\sigma \colon C_0 \to C$.
\end{definition}

Fix a shape system $(C_n,\gamma_n)$ for $C$, with canonical maps $\gamma_{\infty,n} \colon C_n \to C$, and write $\gamma_{m,n} \coloneqq \gamma_{m-1} \circ \cdots \circ \gamma_n \colon C_n \to C_m$ for the composite connecting maps ($m \ge n$, with $\gamma_{n,n} = \id$).
Precomposition with $\gamma_n$ gives structure maps $\gamma_n^*$ on each of the towers $([C_n,D])_n$ and $(\mathrm{H}(C_n,D))_n$.

The following combines Theorem~2.10, Proposition~2.8, Theorem~2.12, Theorem~3.7, and Lemma~2.7 of \cite{Carrion-Schafhauser23}.

\begin{theorem}\label{thm:cs}
	Let $C$ be separable with shape system $(C_n,\gamma_n)$ and let $D$ be a $C^*$-algebra.
	\begin{enumerate}
		\item\label{cs:countable} $[[C,D]]$ is first countable; second countable if $D$
		      is separable.
		\item\label{cs:discrete}
		      For each $n$, the map $\gamma_n^* \colon \mathrm{H}(C_{n+1},D) \to \mathrm{H}(C_n,D)$ factors through a discrete space that is countable when $D$ is separable.
		\item\label{cs:pimsner} {\rm(Pimsner's condition.)}
		      For $x_m, x \in [[C,D]]$, $x_m \to x$ if and only if there is $y \in [[C,C(\Ndag, D)]]$ with $y(m) = x_m$ ($m \in \mathbb N$) and $y(\infty) = x$.
		\item\label{cs:haus} {\rm(Hausdorffization.)}
		      The maps $\gamma_{\infty,n}^*$ induce a homeomorphism $[[C,D]]_{\Hd} \cong \llim(\mathrm{H}(C_n,D), \gamma_n^*)$, where $[[C,D]]_{\Hd}$ is the Kolmogorov (i.e., $T_0$, but in this case, Hausdorff) quotient.
		\item\label{cs:factor}
		      If $\alpha \colon C_0 \to C_1$ is a semiprojective $^*$-homomorphism between separable $C^*$-algebras and $\phi,\psi \colon C_1 \to D$ are $^*$-homomorphisms with $[[\phi]] = [[\psi]]$ in $[[C_1,D]]$, then $\phi \circ \alpha$ and $\psi \circ \alpha$ are homotopic through $^*$-homomorphisms.
	\end{enumerate}
\end{theorem}

A \emph{strong homotopy morphism} $(\underline\phi,\underline h) \colon (C_n,\gamma_n) \to D$ (the case of Dadarlat's \emph{strong maps of systems} \cite[Definition~1.5]{Dadarlat94} in which the target system is the constant system $D$) consists of $^*$-homomorphisms $\phi_n \colon C_n \to D$ and homotopies $h_n \colon C_n \to C([0,1],D)$ with $\ev_0 h_n = \phi_n$ and $\ev_1 h_n = \phi_{n+1} \circ \gamma_n$.
Its \emph{homotopy limit} $\hlim(\underline\phi,\underline h)$, an asymptotic morphism $C \asto D$, and the accompanying notion of homotopy between strong homotopy morphisms are due to Dadarlat \cite[Section~2 and Definition~1.5]{Dadarlat94}.
The homotopy limit is recalled in detail in \cite[Section~1]{Carrion-Schafhauser23}.
Briefly, $\hlim(\underline\phi,\underline h)$ concatenates the homotopies over unit intervals: for $a \in C_n$ and $t \in [m,m+1]$ with $m \ge n$,
\begin{equation}\label{eq:telescope}
	\hlim(\underline\phi,\underline h)_t(a) = h_m\big(\gamma_{m,n}(a)\big)(t-m),
\end{equation}
and this determines the asymptotic morphism up to equivalence \cite[Section~2]{Dadarlat94}.
A strong homotopy morphism is precisely a point of the homotopy inverse limit of the tower of mapping spaces $(\Hom(C_n,D))_n$ with the point-norm topology: a point of each space together with a path from it to the image of the next (cf.\ \cite{Bousfield-Kan72}).
Part~\eqref{dad:bij} of the following theorem says that $[[C,D]]$ is $\pi_0$ of this homotopy limit.

\begin{theorem}[Dadarlat]\label{thm:dad}
	Let $C$ be separable with shape system $(C_n,\gamma_n)$ and let $D$ be a separable $C^*$-algebra.
	\begin{enumerate}
		\item\label{dad:rep}
		      Every asymptotic morphism $C \asto D$ is asymptotically homotopic to $\hlim(\underline\phi,\underline h)$ for some strong homotopy morphism $(\underline\phi,\underline h) \colon (C_n,\gamma_n) \to D$ \cite[Corollary~3.15]{Dadarlat94}.
		      In fact, $[[\hlim(\underline\phi,\underline h) \circ \gamma_{\infty,n}]] = [[\phi_n]]$ in $\mathrm{H}(C_n,D)$ for all $n$ \cite[Proposition~1.5]{Carrion-Schafhauser23}.
		\item\label{dad:bij}
		      The homotopy limit induces a bijection from homotopy classes of strong homotopy morphisms $(C_n,\gamma_n) \to D$ onto $[[C,D]]$ \cite[Theorem~3.5]{Dadarlat94}.
		      In particular, two strong homotopy morphisms with asymptotically homotopic homotopy limits are homotopic \cite[Proposition~3.18]{Dadarlat94}.
	\end{enumerate}
\end{theorem}

We will often identify continuous maps of compact metric spaces into mapping spaces with $^*$-homomorphisms: for a compact metric space $X$, $^*$-homomorphisms $C \to C(X,D)$ correspond to point-norm continuous maps $X \to \Hom(C,D)$.
In particular, writing $SD = C_0((0,1),D)$, a $^*$-homomorphism $C \to SD$ is a based loop at $0$ in $\Hom(C,D)$, a $^*$-homomorphism $C \to C([0,1],SD)$ is a based homotopy of such loops, and a $^*$-homomorphism $C \to C([0,1]^2,D)$ is a square in $\Hom(C,D)$ with prescribed boundary.

When $D$ is stable there are isometries $v_1, v_2$ in the multiplier algebra of $D$ with $v_1 v_1^* + v_2 v_2^* = 1$, and the \emph{orthogonal sum} of $\phi,\psi \in \Hom(C,D)$ is $\phi \oplus \psi \coloneqq v_1 \phi(\,\cdot\,) v_1^* + v_2 \psi(\,\cdot\,) v_2^*$.
Up to homotopy it does not depend on the choice of $v_1$, $v_2$.

\begin{lemma}[cf.\ {\cite[Proposition~25.4.3]{Blackadar98}}]
	\label{lem:loops}
	Let $C$ be separable and let $D$ be a stable $C^*$-algebra.
	Then $[C,SD] = \pi_1(\Hom(C,D),0)$, the loop-concatenation sum agrees with the orthogonal sum, and $[C,SD]$ is an abelian group.
	Concatenation in the suspension coordinate likewise agrees with the orthogonal sum on $[[C,SD]]$, which is an abelian group as well.
\end{lemma}

\begin{proof}
	The identification $[C,SD] = \pi_1(\Hom(C,D),0)$ is a special case of the correspondence above.
	Both the loop sum $\ast$ and the orthogonal sum $\oplus$ are unital binary operations on $[C,SD]$ with common unit $0$, and they satisfy the interchange law $(a \oplus b) \ast (c \oplus d) = (a \ast c) \oplus (b \ast d)$: at suspension parameter $s \le \tfrac12$ both sides are $v_1 a (2s) v_1^* + v_2 b (2s) v_2^*$, and at $s \ge \tfrac12$ both are $v_1 c (2s-1) v_1^* + v_2 d (2s-1) v_2^*$.
	Moreover $a \oplus 0 \simeq a$ and $0 \oplus b \simeq b$ through loops.
	To see this, choose a strictly continuous path of isometries from $1$ to $v_i$ in the multiplier algebra, which exists since $D$ is stable (cf.\ \cite[25.4.1]{Blackadar98}).
	Conjugating along it gives such a homotopy.
	Hence
	\[
		a \ast b = (a \oplus 0) \ast (0 \oplus b) = (a \ast 0) \oplus (0 \ast b) = a \oplus b,
	\]
	so the two sums coincide.
	Commutativity follows in the same way, and inverses are given by loop reversal.
	The statements for asymptotic morphisms are \cite[25.4.1 and Proposition~25.4.3]{Blackadar98}.
\end{proof}

Recall that a \emph{tower} of abelian groups is an inverse sequence $G_1 \xleftarrow{\ g_1\ } G_2 \xleftarrow{\ g_2\ } \cdots$ with structure maps $g_n \colon G_{n+1} \to G_n$, and that, writing $\sh\big((x_n)_n\big) = \big(g_n(x_{n+1})\big)_n$ for the shift on $\prod_n G_n$,
\[
	\llim G_n = \ker(1-\sh)
	\qquad\text{and}\qquad
	\llo G_n = \operatorname{coker}(1-\sh);
\]
see e.g.\ \cite[Section~3]{Schochet03}.
The next lemma, which compares towers that interleave, is folklore---we include the short proof since we could not locate a reference.

\begin{lemma}\label{lem:interleave}
	Let $(G_n,g_n)$ and $(K_n,k_n)$ be towers of abelian groups and suppose there are homomorphisms $q_n \colon G_n \to K_n$ and $r_n \colon K_{n+1} \to G_n$ with
	\[
		r_n \circ q_{n+1} = g_n
		\qquad\text{and}\qquad
		q_n \circ r_n = k_n
	\]
	for all $n$.
	Then $(q_n)_n$ and $(r_n)_n$ induce mutually inverse isomorphisms $\llim G_n \cong \llim K_n$ and $\llo G_n \cong \llo K_n$.
\end{lemma}

\begin{proof}
	The hypotheses say precisely that the alternating tower
	\[
		G_1 \xleftarrow{\ r_1\ } K_2 \xleftarrow{\ q_2\ } G_2 \xleftarrow{\ r_2\ } K_3 \xleftarrow{\ q_3\ } G_3 \xleftarrow{\ r_3\ } \cdots
	\]
	contains both $(G_n,g_n)$ and $(K_{n+1},k_{n+1})$ as cofinal subsequences: consecutive maps compose to $r_n \circ q_{n+1} = g_n$ and $q_{n+1} \circ r_{n+1} = k_{n+1}$.
	Since $\llim$ and $\llo$ are unchanged by passage to cofinal subsequences \cite[Proposition~3.3]{Schochet03}, both towers compute the $\llim$ and $\llo$ of the alternating tower, via the maps induced by $(q_n)_n$ and $(r_n)_n$.
\end{proof}

The isomorphism of Theorem~\ref{thm:cs}\eqref{cs:haus} is stated in terms of $\mathrm{H}(C_n,D)$, the $^*$-homomorphisms modulo asymptotic homotopy, yet the constructions of Section~\ref{sec:milnor} take place in the genuine homotopy classes $[C_n,D]$.
The two need not agree stage by stage (asymptotically homotopic $^*$-homomorphisms out of $C_n$ need not be homotopic, since $C_n$ itself is not assumed semiprojective), but the next lemma shows that the comparison of consecutive stages furnished by Theorem~\ref{thm:cs}\eqref{cs:factor} makes them agree after passing to limits, which is all that the inverse limit and its first derived functor see.

\begin{lemma}\label{lem:compare-towers}
	Let $(C_n,\gamma_n)$ be a shape system for the separable algebra $C$, and let $D$ be the suspension of a stable $C^*$-algebra, so that $[C_n,D]$ and its quotient $\mathrm{H}(C_n,D)$ are abelian groups (Lemma~\ref{lem:loops}).
	The natural surjections $q_n \colon [C_n,D] \to \mathrm{H}(C_n,D)$ and the maps
	\[
		r_n \colon \mathrm{H}(C_{n+1},D) \longrightarrow [C_n,D],\qquad
		[[\phi]]\longmapsto[\phi \circ \gamma_n],
	\]
	satisfy the hypotheses of Lemma~\ref{lem:interleave} for the towers with structure maps $\gamma_n^*$.
	Consequently,
	\[
		\llim[C_n,D] \cong \llim\mathrm{H}(C_n,D)\qquad\text{and}\qquad
		\llo[C_n,D] \cong \llo\mathrm{H}(C_n,D).
	\]
\end{lemma}

\begin{proof}
	If $[[\phi]] = [[\psi]]$ in $\mathrm{H}(C_{n+1},D)$ then $\phi \circ \gamma_n$ and $\psi \circ \gamma_n$ are homotopic through $^*$-homomorphisms by Theorem~\ref{thm:cs}\eqref{cs:factor}, each $\gamma_n$ being semiprojective.
	Therefore, $r_n$ is well-defined, and it is a homomorphism because pre-composition is compatible with the orthogonal sum, as is the quotient map $q_n$.
	For a $^*$-homomorphism $\phi \colon C_{n+1} \to D$ we have $r_n(q_{n+1}[\phi]) = [\phi \circ \gamma_n] = \gamma_n^*[\phi]$ and $q_n(r_n[[\phi]]) = [[\phi \circ \gamma_n]] = \gamma_n^*[[\phi]]$, which are the hypotheses of Lemma~\ref{lem:interleave}.
\end{proof}

To also identify $[[C,D]]_{\Hd}$ with the quotient by the closure of zero, we record that $[[C,D]]$ is a topological group in the generality we need---the argument of \cite[Theorem~4.5]{Carrion-Schafhauser23} for $E(A,B)$ applies unchanged.

\begin{lemma}\label{lem:top-group}
	Let $C$ be separable and let $D$ be the suspension of a stable $C^*$-algebra.
	Then $[[C,D]]$ is a topological abelian group under the orthogonal sum, and its Kolmogorov quotient $[[C,D]]_{\Hd}$ is the quotient $[[C,D]]/\cz$ by the closure of zero.
	In particular $[[C,D]]_{\Hd}$ is Hausdorff.
\end{lemma}

\begin{proof}
	$[[C,D]]$ is an abelian group by Lemma~\ref{lem:loops}.
	For continuity of subtraction, suppose $x_m \to x$ and $y_m \to y$.
	Pimsner's condition (Theorem~\ref{thm:cs}\eqref{cs:pimsner}) provides $\hat x,\hat y \in [[C,C(\Ndag,D)]]$ with $\hat x(m) = x_m$, $\hat y(m) = y_m$ for $m \in \mathbb N$ and $\hat x(\infty) = x$, $\hat y(\infty) = y$.
	Since $C(\Ndag,D)$ is again the suspension of a stable $C^*$-algebra, the difference $\hat x - \hat y$ is defined, and $(\hat x - \hat y)(m) = x_m - y_m$ for $m \in \mathbb N$ and $(\hat x - \hat y)(\infty) = x - y$, the evaluations being suspensions of $^*$-homomorphisms and hence compatible with concatenation in the suspension coordinate, which computes the sum (Lemma~\ref{lem:loops}).
	Thus $x_m - y_m \to x - y$ by Theorem~\ref{thm:cs}\eqref{cs:pimsner} again.
	First countability (Theorem~\ref{thm:cs}\eqref{cs:countable}) upgrades this to continuity of subtraction, so $[[C,D]]$ is a topological group.
	Translations are then homeomorphisms, so two classes are topologically indistinguishable exactly when their difference lies in $\cz$.
	Hence the Kolmogorov quotient is the quotient by $\cz$, and it is Hausdorff because the $T_0$ quotient of a topological group is Hausdorff.
\end{proof}

\begin{proposition}\label{prop:ELlim}
	Let $C$ be separable with shape system $(C_n,\gamma_n)$ and let $D$ be the suspension of a stable $C^*$-algebra.
	There is a surjective homomorphism
	\begin{equation}\label{eq:ELlim}
		\Theta \colon [[C,D]] \longrightarrow \llim[C_n,D]
		\qquad\text{with}\qquad
		\ker\Theta = \cz,
	\end{equation}
	namely the composition of the quotient map onto $[[C,D]]_{\Hd} = [[C,D]]/\cz$ with the isomorphisms $[[C,D]]_{\Hd} \cong \llim\mathrm{H}(C_n,D) \cong \llim[C_n,D]$.
\end{proposition}

\begin{proof}
	Lemma~\ref{lem:top-group} identifies the Kolmogorov quotient $[[C,D]]_{\Hd}$ with $[[C,D]]/\cz$.
	The homeomorphism $[[C,D]]_{\Hd} \cong \llim\mathrm{H}(C_n,D)$ of Theorem~\ref{thm:cs}\eqref{cs:haus} is additive, each $\gamma_{\infty,n}^*$ being additive, hence an isomorphism of groups, and Lemma~\ref{lem:compare-towers} provides the isomorphism $\llim\mathrm{H}(C_n,D) \cong \llim[C_n,D]$.
\end{proof}

$E$-theory and its Hausdorffization are the instances $E(A,B) = [[SA \otimes \K,SB \otimes \K]]$ and $EL(A,B) = E(A,B)/\cz$ \cite[Section~4.2]{Carrion-Schafhauser23}, where Equation~\eqref{eq:ELlim} reads
\[
	EL(A,B) \cong \llim[C_n,SB \otimes \K]
\]
and yields the surjection $\Theta$ and the right-hand term of Theorem~\ref{thm:main}.

\section{The connecting map and the Milnor sequence}\label{sec:milnor}

Throughout this section $C$ is separable with a fixed shape system $(C_n,\gamma_n)$, and $D$ is the suspension of a separable stable $C^*$-algebra, hence itself stable, so that $[[C,D]]$, $[C_n,D]$, and $[C_n,SD]$ are abelian groups with the class of the zero $^*$-homomorphism as neutral element (Lemma~\ref{lem:loops}).
Theorem~\ref{thm:main} will follow from the following more general statement (it is the case $C = SA \otimes \K$, $D = SB \otimes \K$).

\begin{theorem}\label{thm:general}
	Let $C$ be a separable $C^*$-algebra with shape system $(C_n,\gamma_n)$ and let $D$ be the suspension of a separable stable $C^*$-algebra.
	With $\Theta$ as in Proposition~\ref{prop:ELlim}, there is an exact sequence of abelian groups
	\[
		0 \to \llo [C_n,SD] \xrightarrow{\ \partial\ } [[C,D]]
		\xrightarrow{\ \Theta\ } \llim[C_n,D] \to 0 .
	\]
\end{theorem}

\noindent The rest of this section is devoted to its proof.

\subsection{The connecting homomorphism}

We consider the tower $([C_n,SD])_n$ with structure maps $\gamma_n^*$ and, as in Section~\ref{sec:prelim}, realize $\llo[C_n,SD]$ as the cokernel of $1-\sh$ on $\prod_n[C_n,SD]$, where now $\sh\big((\eta_n)_n\big) = \big(\gamma_n^*(\eta_{n+1})\big)_n$.

Let $\eta_n \colon C_n \to SD$ ($n \in \mathbb N$) be $^*$-homomorphisms, i.e., loops at $0$ in $\Hom(C_n,D)$.
The pair $(\underline 0,\underline\eta)$ is then a strong homotopy morphism $(C_n,\gamma_n) \to D$ whose $^*$-homomorphisms are all zero.
Indeed, the homotopy $\eta_n$ is required to run from the zero map to the zero map, which is exactly what being a loop at $0$ says.
Define
\[
	\partial_0 \colon \prod_n[C_n,SD] \longrightarrow [[C,D]],\qquad
	\partial_0\big(([\eta_n])_n\big) \coloneqq \big[\big[\hlim(\underline 0,\underline\eta)\big]\big];
\]
see Figure~\ref{fig:partial}.
This does not depend on the choice of representatives.
For this, note that if $\eta_n \simeq \eta_n'$ through loops for every $n$, then these homotopies, squares in $\Hom(C_n,D)$ whose other two edges are constant at $0$, are exactly the data of a homotopy of strong homotopy morphisms from $(\underline 0,\underline\eta)$ to $(\underline 0,\underline{\eta'})$, so the two homotopy limits agree in $[[C,D]]$ by Theorem~\ref{thm:dad}\eqref{dad:bij}.

\begin{figure}[ht]
	\centering
	\begin{tikzpicture}[font=\footnotesize]
		\def\r{1.5pt}
		\def\ph{1.22}
		\draw[-{Computer Modern Rightarrow}] (-0.7,0) -- (6.5,0) node[right] {$t$};
		\foreach \x/\lab in {0/1, 1.8/2, 3.6/3, 5.4/4} {
				\fill (\x,0) circle (\r);
				\node[below=2pt] at (\x,0) {$0$};
				\node[below=12pt] at (\x,0) {\scriptsize $t=\lab$};
			}
		\draw (0,0)   to[out=90,in=90,looseness=1.7] (1.8,0);
		\draw (1.8,0) to[out=90,in=90,looseness=1.7] (3.6,0);
		\draw (3.6,0) to[out=90,in=90,looseness=1.7] (5.4,0);
		\node at (0.9,\ph) {$\eta_1$};
		\node at (2.7,\ph) {$\eta_2$};
		\node at (4.5,\ph) {$\eta_3$};
		\node at (6.0,0.1) {$\cdots$};
		\node at (5.95,0.62) {\scriptsize $t\to\infty$};
	\end{tikzpicture}
	\caption{The connecting map $\partial$: the loop $\eta_n$ runs over $t \in [n,n+1]$, returning to $0$ at each integer, so $\partial_0(\eta)$ vanishes on every stage $C_n$ yet can represent a nonzero class of $[[C,D]]$.}
	\label{fig:partial}
\end{figure}
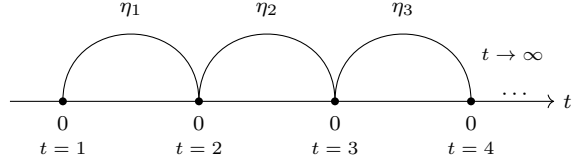

\begin{lemma}\label{lem:partial-hom}
	$\partial_0 \colon \prod_n[C_n,SD] \to [[C,D]]$ is a group homomorphism, and $\Theta \circ \partial_0 = 0$.
\end{lemma}

\begin{proof}
	Additivity: the orthogonal sum of two homotopy limits is the homotopy limit of the levelwise orthogonal sum of the data (immediate from Equation~\eqref{eq:telescope}), and choosing the orthogonal sums $\eta_n \oplus \eta_n'$ as representatives of $(\eta+\eta')_n$ (Lemma~\ref{lem:loops}) makes the homotopy data of $\partial_0(\eta) + \partial_0(\eta')$ and of $\partial_0(\eta+\eta')$ coincide.
	For $\Theta \circ \partial_0 = 0$: by Theorem~\ref{thm:dad}\eqref{dad:rep}, $[[\hlim(\underline 0,\underline\eta) \circ \gamma_{\infty,n}]] = [[0]] = 0$ in $\mathrm{H}(C_n,D)$ for all $n$, and by its construction in Proposition~\ref{prop:ELlim} the element $\Theta(\partial_0(\eta))$ is computed from these restrictions.
	Therefore, it vanishes.
\end{proof}

The key point is the following shift-invariance, where the reparametrization-invariance of asymptotic morphisms enters.

\begin{lemma}\label{lem:shift}
	$\partial_0 \circ \sh = \partial_0$.
	Consequently, $\partial_0 \circ (1-\sh) = 0$ and $\partial_0$ descends to a homomorphism
	\[
		\partial \colon \llo [C_n,SD] \longrightarrow \cz \subseteq [[C,D]].
	\]
\end{lemma}

\begin{proof}
	Fix $^*$-homomorphisms $\zeta_n \colon C_n \to SD$ ($n \in \mathbb N$), write $\underline\zeta = (\zeta_n)_n$, and set $\psi = \hlim(\underline 0,\underline\zeta)$, realized via Equation~\eqref{eq:telescope}: for $a \in C_n$ and $t \in [m,m+1]$ ($m \ge n$), $\psi_t(a) = \zeta_m(\gamma_{m,n}(a))(t-m)$, where $\zeta_m(\,\cdot\,)(s) \in D$ denotes the loop $\zeta_m$ at parameter $s \in [0,1]$ (vanishing at $s = 0,1$).
	The shifted sequence is $(\sh\zeta)_m = \gamma_m^*\zeta_{m+1} = \zeta_{m+1} \circ \gamma_m$, so for $a \in C_n$ and $t \in [m,m+1]$,
	\begin{align*}
		\big(\hlim(\underline 0,\underline{\sh\zeta})\big)_t(a)
		 & = \zeta_{m+1}\big(\gamma_m(\gamma_{m,n}(a))\big)(t-m) \\
		 & = \zeta_{m+1}\big(\gamma_{m+1,n}(a)\big)(t-m)
		= \psi_{t+1}(a),
	\end{align*}
	using $\gamma_m \circ \gamma_{m,n} = \gamma_{m+1,n}$ and the same formula for $\psi$ on
	$[m+1,m+2]$.
	Thus $\hlim(\underline 0,\underline{\sh\zeta})$ is the reparametrization $t \mapsto t+1$ of $\psi$, hence asymptotically homotopic to $\psi$ \cite[25.1.2(h)]{Blackadar98}.
	Hence $\partial_0(\sh\zeta) = [[\psi]] = \partial_0(\zeta)$.

	Since $\partial_0$ is a homomorphism (Lemma~\ref{lem:partial-hom}), $\partial_0 \circ (1-\sh) = \partial_0 - \partial_0 \circ \sh = 0$.
	As $\llo[C_n,SD]$ is the cokernel of $1-\sh$, $\partial_0$ descends to $\partial$.
	Finally, $\operatorname{im}\partial \subseteq \cz$ because $\Theta \circ \partial_0 = 0$ (Lemma~\ref{lem:partial-hom}).
\end{proof}

\subsection{Exactness}

Now we show that $\partial$ is injective and that its image is all of $\cz$.
Surjectivity comes first: from the homotopy data of a representation of a class in $\cz$ as a homotopy limit, we build an explicit preimage under $\partial$.

\begin{lemma}\label{lem:surj}
	Let $\psi = \hlim(\underline\phi,\underline h)$ be the homotopy limit of a strong homotopy morphism $(\underline\phi,\underline h) \colon (C_n,\gamma_n) \to D$ and suppose $[[\psi]] \in \cz$.
	Then each $\phi_n$ is null-homotopic through $^*$-homomorphisms, and for any null-homotopies $k_n \colon C_n \to C([0,1],D)$ with $\ev_0 k_n = \phi_n$ and $\ev_1 k_n = 0$, the concatenations
	\begin{equation}\label{eq:Lambda0}
		\Lambda_n \coloneqq \overline{k_n} \ast h_n \ast (k_{n+1} \circ \gamma_n)
	\end{equation}
	are loops at $0$, and $\partial_0\big(([\Lambda_n])_n\big) = [[\psi]]$.
	In particular $\partial$ maps $\llo[C_n,SD]$ onto $\cz$.
\end{lemma}

\begin{proof}
	Since $[[\psi]] \in \cz = \ker\Theta$, the construction of $\Theta$ in Proposition~\ref{prop:ELlim} gives $[[\phi_n]] = [[\psi \circ \gamma_{\infty,n}]] = 0$ in $\mathrm{H}(C_n,D)$ for all $n$, using Theorem~\ref{thm:dad}\eqref{dad:rep} for the first equality.
	This records only that each $\phi_n$ is asymptotically null-homotopic, but in fact $\phi_n$ is null-homotopic through $^*$-homomorphisms.
	This is because $\gamma_n$ is semiprojective and $[[\phi_{n+1}]] = 0$, so Theorem~\ref{thm:cs}\eqref{cs:factor} provides a homotopy through $^*$-homomorphisms from $\phi_{n+1} \circ \gamma_n$ to $0$, and concatenating the homotopy $h_n$ (from $\phi_n$ to $\phi_{n+1} \circ \gamma_n$) with it gives one from $\phi_n$ to $0$.
	Fix null-homotopies $k_n$ as in the statement.
	In Equation~\eqref{eq:Lambda0}, $\overline{k_n}$ runs from $\ev_1 k_n = 0$ to $\ev_0 k_n = \phi_n$; then $h_n$ from $\phi_n$ to $\phi_{n+1}\gamma_n$; then $k_{n+1} \circ \gamma_n$ from $\ev_0(k_{n+1}\gamma_n) = \phi_{n+1}\gamma_n$ to $0$.
	Thus $\Lambda_n$ is a loop at $0$, i.e., a $^*$-homomorphism $C_n \to SD$.

	Now build a strong homotopy morphism $(\underline\Phi,\underline H) \colon (C_n) \to C([0,1],D)$ with $\Phi_n \coloneqq k_n$ (the path $\phi_n\rightsquigarrow 0$) and $H_n$ a filling of the square
	\[
		\begin{tikzcd}
			0 \arrow[r, "\Lambda_n"] & 0 \\
			\phi_n \arrow[u, "k_n"] \arrow[r, "h_n"'] \arrow[ur, phantom, "H_n" description] & \phi_{n+1}\gamma_n \arrow[u, "k_{n+1}\gamma_n"']
		\end{tikzcd}
	\]
	whose horizontal edges ($h_n$ below, $\Lambda_n$ above) run along the homotopy parameter of $H_n$, and whose vertical edges ($k_n$ and $k_{n+1} \gamma_n$) along the $[0,1]$-coordinate of the target $C([0,1],D)$.
	The filling exists because the two ways around the boundary are homotopic with endpoints fixed, since $k_n \ast \overline{k_n}$ cancels:
	\[
		k_n \ast \Lambda_n = k_n \ast \overline{k_n} \ast h_n \ast (k_{n+1}\gamma_n)
		\simeq h_n \ast (k_{n+1} \circ \gamma_n),
	\]
	and such a homotopy, suitably reparametrized, is exactly such a filling.
	Its ends are $\ev_0(\underline\Phi,\underline H) = (\underline\phi,\underline h)$ and $\ev_1(\underline\Phi,\underline H) = (\underline 0,\underline{\Lambda})$, so by functoriality of $\hlim$ under post-composition with the evaluations $\ev_t \colon C([0,1],D) \to D$ the asymptotic morphism $\hlim(\underline\Phi,\underline H)$ is a homotopy from $\psi = \hlim(\underline\phi,\underline h)$ to $\hlim(\underline 0,\underline{\Lambda})$.
	Hence $\partial_0\big(([\Lambda_n])_n\big) = \big[\big[\hlim(\underline 0,\underline{\Lambda})\big]\big] = [[\psi]]$.
	Finally, every class in $\cz$ is the class of such a homotopy limit (Theorem~\ref{thm:dad}\eqref{dad:rep}), and $\partial_0$ descends to $\partial$.
	Thus $\partial$ maps onto $\cz$.
\end{proof}

\begin{lemma}\label{lem:inj}
	$\partial$ is injective.
\end{lemma}

\begin{proof}
	Recall from Lemma~\ref{lem:loops} that, with $D$ stable, $[C_n,SD] = \pi_1(\mathrm{Hom}(C_n,D),0)$ is an abelian group under concatenation.
	We compute in the fundamental groupoid of $\mathrm{Hom}(C_n,D)$, writing $\ast$ for concatenation, $\overline{(\,\cdot\,)}$ for reversal, and $\mathrm{const}$ for a constant path.

	Let $(\zeta_n)$ represent a class in $\ker\partial$, so that $\hlim(\underline 0,\underline\zeta)$ is asymptotically null-homotopic.
	This null-homotopy is an asymptotic morphism $C \asto C([0,1],D)$.
	By Theorem~\ref{thm:dad}\eqref{dad:rep} it is asymptotically homotopic to $\hlim(\underline\Phi,\underline H)$ for a strong homotopy morphism $(\underline\Phi,\underline H) \colon (C_n) \to C([0,1],D)$.
	By functoriality of $\hlim$ under $\ev_t \colon C([0,1],D) \to D$, its ends $(\ev_0\Phi_n,\ev_0 H_n)$ and $(\ev_1\Phi_n,\ev_1 H_n)$ are strong homotopy morphisms whose homotopy limits are asymptotically homotopic to $\hlim(\underline 0,\underline\zeta)$ and to $0$, respectively.
	By Theorem~\ref{thm:dad}\eqref{dad:bij}, these ends are then homotopic, as strong homotopy morphisms, to $(\underline 0,\underline\zeta)$ and to the zero morphism.
	Such homotopies provide, for every $n$,
	\begin{itemize}
		\item a path $a_n \colon \ev_0\Phi_n \to 0$ with
		      $\ev_0 H_n \ast (a_{n+1} \circ \gamma_n) \simeq a_n \ast \zeta_n$, and
		\item a path $c_n \colon \ev_1\Phi_n \to 0$ with
		      $\ev_1 H_n \ast (c_{n+1} \circ \gamma_n) \simeq c_n$.
	\end{itemize}
	Indeed, each homotopy is itself a strong homotopy morphism into $C([0,1],D)$: its maps are the paths $a_n$ in the first case and $c_n$ in the second, and its connecting squares are fillings whose two ways around give precisely these two relations.
	Writing $\Phi_n$ also for the path $s \mapsto \ev_s\Phi_n$ from $\ev_0\Phi_n$ to $\ev_1\Phi_n$, the connecting square $H_n$ of $(\underline\Phi,\underline H)$ gives
	\begin{equation}\label{eq:Hsquare}
		(\ev_0 H_n) \ast (\Phi_{n+1} \circ \gamma_n) \simeq \Phi_n \ast (\ev_1 H_n).
	\end{equation}
	Put $p_n \coloneqq \overline{a_n} \ast \Phi_n \ast c_n$, a loop at $0$, i.e., an element of $[C_n,SD]$.
	Then $p_n - \gamma_n^*p_{n+1}$ is represented by
	\[
		\overline{a_n} \ast \Phi_n \ast c_n \ast \overline{(c_{n+1}\gamma_n)}
		\ast \overline{(\Phi_{n+1}\gamma_n)} \ast (a_{n+1}\gamma_n)
		\simeq \overline{a_n} \ast (\ev_0 H_n) \ast (a_{n+1}\gamma_n) \simeq \zeta_n.
	\]
	The first $\simeq$ substitutes $\overline{(\Phi_{n+1}\gamma_n)} \simeq \overline{(\ev_1 H_n)} \ast \overline{\Phi_n} \ast (\ev_0 H_n)$ from Equation~\eqref{eq:Hsquare}, cancels $c_n \ast \overline{(c_{n+1}\gamma_n)} \ast \overline{(\ev_1 H_n)} \simeq \mathrm{const}$ (using $\overline{\ev_1 H_n} \simeq (c_{n+1}\gamma_n) \ast \overline{c_n}$), and then cancels $\Phi_n \ast \overline{\Phi_n} \simeq \mathrm{const}$; the second substitutes $\ev_0 H_n \simeq a_n \ast \zeta_n \ast \overline{(a_{n+1}\gamma_n)}$.
	Hence $\zeta = (1-\sh)(p)$ and $[(\zeta_n)] = 0$ in $\llo[C_n,SD]$.
\end{proof}

\begin{proposition}\label{prop:exact}
	$\Theta$ is surjective with $\ker\Theta = \cz$, and
	$\partial \colon \llo[C_n,SD] \to \cz$ is an isomorphism whose inverse sends $[[\psi]] \in \cz$
	to the class of $([\Lambda_n])_n$ in $\llo[C_n,SD]$ constructed in Lemma~\ref{lem:surj}.
	Therefore, the sequence of Theorem~\ref{thm:general} is exact.
\end{proposition}

\begin{proof}
	Surjectivity of $\Theta$ and $\ker\Theta = \cz$ follow from Proposition~\ref{prop:ELlim}.
	By Lemma~\ref{lem:shift} $\partial$ is a well-defined homomorphism into $\cz$, surjective by Lemma~\ref{lem:surj} and injective by Lemma~\ref{lem:inj}.
	Hence it is an isomorphism onto $\cz$, with inverse as stated (the class of $([\Lambda_n])_n$ is independent of all the choices made in its construction, $\partial$ being injective).
	Thus $\operatorname{im}\partial = \cz = \ker\Theta$ with $\partial$ injective and $\Theta$ surjective, i.e., the sequence of Theorem~\ref{thm:general} is exact.
\end{proof}

\begin{proof}[Proof of Theorems~\ref{thm:general} and \ref{thm:main}]
	Theorem~\ref{thm:general} combines Lemmas~\ref{lem:partial-hom}--\ref{lem:inj} and
	Proposition~\ref{prop:exact}.
	Theorem~\ref{thm:main} is the case $C = SA \otimes \K$, $D = SB \otimes \K$ ($SD = S^2B \otimes \K$, $[[C,D]] = E(A,B)$), with Proposition~\ref{prop:ELlim}.
\end{proof}

\section{Naturality and comparisons}\label{sec:examples}

In this final section $A$ and $B$ are separable $C^*$-algebras and $(C_n,\gamma_n)$ is a shape system for $SA \otimes \K$.
We record the naturality of the Milnor sequence, a Hausdorffness criterion and the resulting dichotomy, and the comparisons with the sequences of Willett--Yu and Schochet.

\begin{proposition}\label{prop:natural}
	The outer terms of the sequence of Theorem~\ref{thm:main} do not depend on the chosen shape system, and the sequence is natural in $A$ and $B$.
\end{proposition}

\begin{proof}
	Theorem~\ref{thm:main} identifies both outer terms with objects intrinsic to the pair $(A,B)$: for every shape system of $SA \otimes \K$, Proposition~\ref{prop:ELlim} identifies $\llim[C_n,SB \otimes \K]$ with $EL(A,B)$, and the theorem identifies $\llo[C_n,S^2B \otimes \K]$ with $\cz$.
	The sequence is natural in $B$ by post-composition.
	For naturality in $A$, let $A'$ be a separable $C^*$-algebra and $(C_n',\gamma_n')$ a shape system for $SA' \otimes \K$.
	A $^*$-homomorphism $f \colon A \to A'$ induces $Sf \otimes \id_\K \colon SA \otimes \K \to SA' \otimes \K$, represented by a strong map of systems $(C_n,\gamma_n) \to (C_n',\gamma_n')$, the system-to-system notion underlying the strong homotopy morphisms of Section~\ref{sec:prelim} \cite[Definition~1.5 and Theorem~3.5]{Dadarlat94}.
	Indeed, after reindexing, it consists of $^*$-homomorphisms $\xi_n \colon C_n \to C_{m(n)}'$ with homotopies filling the squares $\xi_{n+1} \circ \gamma_n \simeq \gamma_{m(n+1),m(n)}' \circ \xi_n$, compatibly with $Sf \otimes \id_\K$ under the canonical maps.
	On homotopy classes the filled squares commute, so the maps $\xi_n^*$ form a morphism from the reindexed towers over $(C_n')$ to the towers over $(C_n)$ and induce maps on $\llim$ and $\llo$, reindexing being cofinal.
	The square with $\Theta$ commutes by the same compatibility, and the square with $\partial$ commutes because pre-composing $\hlim(\underline 0,\underline\eta)$ with $Sf \otimes \id_\K$ is again a homotopy limit, computed from the pulled-back loop data by functoriality of the homotopy limit under composition of strong maps \cite[Theorem~2.4]{Dadarlat94}.
	On $\llo$, the reindexing replaces each loop by the concatenation of the pulled-back loops over the skipped stages, which is the isomorphism induced by passage to a cofinal subsequence.
\end{proof}

\begin{corollary}\label{cor:phantom}
	Let $A$ and $B$ be separable $C^*$-algebras and let $(C_n,\gamma_n)$ be a shape system for $SA \otimes \K$, with canonical maps $\gamma_{\infty,n}$.
	A class $x \in E(A,B)$ lies in $\cz$ if and only if $\gamma_{\infty,n}^*(x) = 0$ in $\mathrm{H}(C_n,SB \otimes \K)$ for every $n$: the phantom classes are exactly the classes that vanish on every stage of a shape system.
\end{corollary}

\begin{proof}
	By Proposition~\ref{prop:ELlim}, $\ker\Theta = \cz$, and the components of $\Theta$ are induced by the maps $\gamma_{\infty,n}^*$.
\end{proof}

For the next two results, $B$ being separable, we may assume without loss of generality that the towers at hand consist of countable groups.
Indeed, by Theorem~\ref{thm:cs}\eqref{cs:discrete} each structure map of the tower $(\mathrm{H}(C_n,S^2B \otimes \K))_n$ factors through a countable set, so the groups $\Gamma_n \coloneqq \operatorname{im}\gamma_n^* \subseteq \mathrm{H}(C_n,S^2B \otimes \K)$ are countable.
The inclusions $\Gamma_n \to \mathrm{H}(C_n,S^2B \otimes \K)$ and the corestrictions $\mathrm{H}(C_{n+1},S^2B \otimes \K) \to \Gamma_n$ of $\gamma_n^*$ satisfy the hypotheses of Lemma~\ref{lem:interleave} (there, take $G_n = \Gamma_n$ and $K_n = \mathrm{H}(C_n,S^2B \otimes \K)$, with the inclusions as the maps $q_n$ and the corestrictions as the maps $r_n$), so together with Lemma~\ref{lem:compare-towers} we obtain
\begin{equation}\label{eq:imtower}
	\llo [C_n,S^2B \otimes \K] \cong \llo \Gamma_n,
\end{equation}
the $\llo$ of a tower of countable abelian groups.

Recall that a tower $(G_n,g_n)$ satisfies the \emph{Mittag--Leffler condition} if for each $n$ the descending chain of subgroups $\operatorname{im}(G_m \to G_n) \subseteq G_n$, $m \ge n$, is eventually constant: there is $m_0 \ge n$ with $\operatorname{im}(G_m \to G_n) = \operatorname{im}(G_{m_0} \to G_n)$ for all $m \ge m_0$.

\begin{corollary}\label{cor:hausdorff}
	Let $A$ and $B$ be separable $C^*$-algebras and let $(C_n,\gamma_n)$ be a shape system for $SA \otimes \K$.
	Then $E(A,B)$ is Hausdorff, that is $E(A,B) = EL(A,B)$, if and only if $\llo[C_n,S^2B \otimes \K] = 0$, if and only if the tower $(\Gamma_n)_n$ of Equation~\eqref{eq:imtower} is Mittag--Leffler.
	In particular $E(A,B)$ is Hausdorff whenever the tower $([C_n,S^2B \otimes \K])_n$ is Mittag--Leffler.
\end{corollary}

\begin{proof}
	The first equivalence is Theorem~\ref{thm:main}, $E(A,B)$ being Hausdorff exactly when $\cz = 0$.
	A Mittag--Leffler tower of abelian groups has vanishing $\llo$, and for towers of countable abelian groups the converse holds as well \cite[Theorem~3.11]{Schochet03}.
	By Equation~\eqref{eq:imtower} this gives the second equivalence.
	For the final statement: if the tower $([C_n,S^2B \otimes \K])_n$ is Mittag--Leffler, then its $\llo$ vanishes, and by Theorem~\ref{thm:main} so does $\cz$.
\end{proof}

\begin{remark}
	\label{rem:dichotomy}
	The $\llo$ of a tower of countable abelian groups is either trivial or uncountable, by an argument of Gray \cite[p.~242]{Gray66} (see also \cite[Proposition~3.5]{Schochet03}).
	By Equation~\eqref{eq:imtower}, $\cz$ is therefore either trivial or uncountable: if $E(A,B)$ is not Hausdorff, it contains uncountably many phantom classes.
	This is the $E$-theoretic counterpart of \cite[Corollary~7.9]{Willett-Yu21}, where the observation is credited to Schochet.
\end{remark}

\begin{proposition}[Comparison with the Willett--Yu sequence]\label{prop:compare}
	Let $A$ and $B$ be separable $C^*$-algebras with $A$ nuclear, and let $(C_n,\gamma_n)$ be a shape system for $SA \otimes \K$.
	The isomorphism $E(A,B) \cong KK(A,B)$, an isomorphism of topological groups for nuclear $A$, carries $\cz$ onto $\overline{\{0\}}_{KK(A,B)}$ and induces $EL(A,B) \cong KL(A,B)$, so Theorem~\ref{thm:main} becomes the sequence
	\begin{equation}\label{eq:KLext}
		0 \longrightarrow \overline{\{0\}}_{KK(A,B)} \longrightarrow KK(A,B) \longrightarrow KL(A,B) \longrightarrow 0 .
	\end{equation}
	The Willett--Yu controlled-$KK$ Milnor sequence is isomorphic to it as an extension of $KL(A,B)$ by $\overline{\{0\}}_{KK(A,B)}$; in particular
	\[
		\llo [C_n,S^2B \otimes \K] \cong \overline{\{0\}}_{KK(A,B)} \cong \llo KK_{\varepsilon_n}(X_n,SB),
	\]
	both being the kernel of $KK(A,B) \to KL(A,B)$.
	Here $(X_n, \varepsilon_n)$ is a cofinal sequence of finite subsets $X_n \subseteq A$ and tolerances $\varepsilon_n$ for the controlled $KK$-groups of \cite[Sections~6 and~7]{Willett-Yu21}.
\end{proposition}

\begin{proof}
	For $A$ nuclear the canonical natural transformation $KK(A,B) \to E(A,B)$ is an isomorphism \cite[Corollary~9]{Connes-Higson90}.
	It carries $\overline{\{0\}}_{KK(A,B)}$ onto $\cz$ and induces $KL(A,B) \cong EL(A,B)$ \cite[Section~4.2]{Carrion-Schafhauser23}.
	Under these identifications the sequence of Theorem~\ref{thm:main} becomes the extension~\eqref{eq:KLext}, with $\llo [C_n,S^2B \otimes \K] \cong \overline{\{0\}}_{KK(A,B)}$, $\llim [C_n,SB \otimes \K] \cong KL(A,B)$, and $\Theta$ the canonical surjection $KK(A,B) \to KL(A,B)$.
	Willett--Yu prove the same two identifications for their controlled tower ($\llim KK_{\varepsilon_n}(X_n,B) \cong KL(A,B)$ and $\llo KK_{\varepsilon_n}(X_n,SB) \cong \overline{\{0\}}_{KK(A,B)}$), compatibly with the same surjection \cite[Theorems~6.14 and~7.8]{Willett-Yu21}.
	Then the two Milnor sequences are isomorphic as extensions and their $\varprojlim^1$ terms agree.
\end{proof}

The Willett--Yu sequence, like ours, needs no nuclearity: their identifications hold for arbitrary separable $A$ and $B$ \cite[Theorems~6.14 and~7.8]{Willett-Yu21}.
The hypothesis in Proposition~\ref{prop:compare} serves only to identify $E(A,B)$ with $KK(A,B)$ (and for this much less is needed, see e.g.\ \cite[Theorem~25.6.3]{Blackadar98}).
In general, the two sequences describe the closures of zero of potentially different groups.

\begin{proposition}[The UCT case]\label{prop:uct}
	Let $A$ and $B$ be separable $C^*$-algebras, suppose $A$ satisfies the UCT, and let $(C_n,\gamma_n)$ be a shape system for $SA \otimes \K$.
	Then
	\[
		\llo [C_n,S^2B \otimes \K] \cong \mathrm{Pext}^1_{\mathbb Z}\big(K_{*+1}(A),K_*(B)\big)
	\]
	and $\llim[C_n,SB \otimes \K] \cong \operatorname{Hom}_\Lambda(\underline K(A),\underline K(B))$.
\end{proposition}

\begin{proof}
	A separable $C^*$-algebra satisfies the UCT for all $B$ if and only if it is $KK$-equivalent to a commutative $C^*$-algebra, by the argument of \cite[Corollary~7.5]{Rosenberg-Schochet87} (cf.\ the discussion following Remark~7.6 there).
	The canonical natural transformation $KK \to E$ (see \cite[Corollary~9]{Connes-Higson90} or \cite[Corollary~25.5.8]{Blackadar98}) is compatible with the composition products, so it carries $KK$-equivalences to $E$-equivalences.
	Hence $A$ is $E$-equivalent to a commutative $C^*$-algebra.
	For such $A$, it is shown in \cite[Section~4.2 and Theorem~4.8]{Carrion-Schafhauser23} that $\cz$ is the image of the subgroup of pure extensions under the injective map $\mathrm{Ext}^1_{\mathbb Z}(K_{*+1}(A),K_*(B)) \to E(A,B)$ of the $E$-theoretic UCT, and that $EL(A,B) \cong \operatorname{Hom}_\Lambda(\underline K(A),\underline K(B))$.
	The statement now follows from Theorem~\ref{thm:main} and Proposition~\ref{prop:ELlim}.
\end{proof}

Besides the bootstrap class of \cite{Rosenberg-Schochet87}, the hypothesis of Proposition~\ref{prop:uct} covers non-nuclear algebras.
A free group $\mathbb{F}_n$ is a-T-menable, so its reduced group $C^*$-algebra $C^*_r(\mathbb{F}_n)$ satisfies the UCT \cite[Theorem~9.3 and Proposition~10.7]{Tu99}.
For $n \ge 2$ this algebra is not nuclear.

\subsection*{Acknowledgments}
The author is grateful to Christopher Schafhauser for many discussions about the topology on $E$-theory, and to Rufus Willett for enlightening conversations about controlled $KK$-theory and its Milnor exact sequence.

\bibliographystyle{amsalpha}
\bibliography{general-references}

\providecommand{\bysame}{\leavevmode\hbox to3em{\hrulefill}\thinspace}
\providecommand{\MR}{\relax\ifhmode\unskip\space\fi MR }
\providecommand{\MRhref}[2]{%
  \href{http://www.ams.org/mathscinet-getitem?mr=#1}{#2}
}
\providecommand{\href}[2]{#2}
\begin{thebibliography}{GHT00}

\bibitem[BK72]{Bousfield-Kan72}
Aldridge~K. Bousfield and Daniel~M. Kan, \emph{Homotopy limits, completions and
  localizations}, Lecture Notes in Mathematics, vol. 304, Springer-Verlag,
  Berlin-New York, 1972. \MR{365573}

\bibitem[Bla85]{Blackadar85}
Bruce Blackadar, \emph{Shape theory for {$C^*$}-algebras}, Math. Scand.
  \textbf{56} (1985), no.~2, 249--275. \MR{813640}

\bibitem[Bla98]{Blackadar98}
\bysame, \emph{{$K$}-theory for operator algebras}, second ed., Mathematical
  Sciences Research Institute Publications, vol.~5, Cambridge University Press,
  Cambridge, 1998. \MR{1656031}

\bibitem[Bro75]{Brown75}
Lawrence~G. Brown, \emph{Operator algebras and algebraic {$K$}-theory}, Bull.
  Amer. Math. Soc. \textbf{81} (1975), no.~6, 1119--1121. \MR{0383090}

\bibitem[Bro77]{Brown77b}
\bysame, \emph{Characterizing {${\rm Ext}(X)$}}, {$K$}-theory and operator
  algebras ({P}roc. {C}onf., {U}niv. {G}eorgia, {A}thens, {G}a., 1975),
  Springer, Berlin, 1977, pp.~10--18. Lecture Notes in Math., Vol. 575.
  \MR{0474275}

\bibitem[Bro84]{Brown84}
\bysame, \emph{The universal coefficient theorem for {${\rm Ext}$} and
  quasidiagonality}, Operator algebras and group representations, {V}ol. {I}
  ({N}eptun, 1980), Monogr. Stud. Math., vol.~17, Pitman, Boston, MA, 1984,
  pp.~60--64. \MR{731763}

\bibitem[CH90]{Connes-Higson90}
Alain Connes and Nigel Higson, \emph{D\'{e}formations, morphismes asymptotiques
  et {$K$}-th\'{e}orie bivariante}, C. R. Acad. Sci. Paris S\'{e}r. I Math.
  \textbf{311} (1990), no.~2, 101--106. \MR{1065438}

\bibitem[CS24]{Carrion-Schafhauser23}
Jos\'e~R. Carri\'on and Christopher Schafhauser, \emph{A topology on
  {$E$}-theory}, J. Lond. Math. Soc. (2) \textbf{109} (2024), no.~6, Paper No.
  e12917, 32. \MR{4751861}

\bibitem[Dad94]{Dadarlat94}
Marius Dadarlat, \emph{Shape theory and asymptotic morphisms for
  {$C^*$}-algebras}, Duke Math. J. \textbf{73} (1994), no.~3, 687--711.
  \MR{1262931}

\bibitem[Dad05]{Dadarlat05}
\bysame, \emph{On the topology of the {K}asparov groups and its applications},
  J. Funct. Anal. \textbf{228} (2005), no.~2, 394--418. \MR{2175412}

\bibitem[GHT00]{GHT00}
Erik Guentner, Nigel Higson, and Jody Trout, \emph{Equivariant {$E$}-theory for
  {$C^*$}-algebras}, Mem. Amer. Math. Soc. \textbf{148} (2000), no.~703,
  viii+86. \MR{1711324}

\bibitem[Gra66]{Gray66}
Brayton~I. Gray, \emph{Spaces of the same {$n$}-type, for all {$n$}}, Topology
  \textbf{5} (1966), 241--243. \MR{196743}

\bibitem[R\o95]{Rordam95}
Mikael R\o{}rdam, \emph{Classification of certain infinite simple
  {$C^*$}-algebras}, J. Funct. Anal. \textbf{131} (1995), no.~2, 415--458.
  \MR{1345038}

\bibitem[RS87]{Rosenberg-Schochet87}
Jonathan Rosenberg and Claude~L. Schochet, \emph{The {K}\"unneth theorem and
  the universal coefficient theorem for {K}asparov's generalized
  {$K$}-functor}, Duke Math. J. \textbf{55} (1987), no.~2, 431--474.
  \MR{894590}

\bibitem[Sch96]{Schochet96}
Claude~L. Schochet, \emph{The {UCT}, the {M}ilnor sequence, and a canonical
  decomposition of the {K}asparov groups}, $K$-Theory \textbf{10} (1996),
  no.~1, 49--72. \MR{1373818}

\bibitem[Sch98]{Schochet98}
\bysame, \emph{Correction to: ``{T}he {UCT}, the {M}ilnor sequence, and a
  canonical decomposition of the {K}asparov groups'' [{$K$}-{T}heory {\bf 10}
  (1996), no.~1, 49--72; {MR}1373818 (97d:46088)]}, $K$-Theory \textbf{14}
  (1998), no.~2, 197--199. \MR{1628275}

\bibitem[Sch01]{Schochet01}
\bysame, \emph{The fine structure of the {K}asparov groups. {I}. {C}ontinuity
  of the {$KK$}-pairing}, J. Funct. Anal. \textbf{186} (2001), no.~1, 25--61.
  \MR{1863291}

\bibitem[Sch02]{Schochet02}
\bysame, \emph{The fine structure of the {K}asparov groups. {II}.
  {T}opologizing the {UCT}}, J. Funct. Anal. \textbf{194} (2002), no.~2,
  263--287. \MR{1934604}

\bibitem[Sch03]{Schochet03}
\bysame, \emph{A {P}ext primer: pure extensions and {$\lim^1$} for infinite
  abelian groups}, New York Journal of Mathematics. NYJM Monographs, vol.~1,
  State University of New York, University at Albany, Albany, NY, 2003,
  Available electronically at \url{http://nyjm.albany.edu/m/2003/1v.pdf}.
  \MR{1993677}

\bibitem[Tu99]{Tu99}
Jean-Louis Tu, \emph{La conjecture de {B}aum-{C}onnes pour les feuilletages
  moyennables}, $K$-Theory \textbf{17} (1999), no.~3, 215--264. \MR{1703305}

\bibitem[WY24]{Willett-Yu24}
Rufus Willett and Guoliang Yu, \emph{The universal coefficient theorem for
  {$C^*$}-algebras with finite complexity}, Mem. Eur. Math. Soc., vol.~8, EMS
  Press, Berlin, 2024. \MR{4711665}

\bibitem[WY26]{Willett-Yu21}
\bysame, \emph{Controlled {$KK$}-theory and a {M}ilnor exact sequence}, Doc.
  Math. \textbf{31} (2026), no.~3, 503--581. \MR{5061886}

\end{thebibliography}

\end{document}